\documentclass{amsart}
\usepackage{amsmath}
\usepackage{amsfonts}
\usepackage{amssymb}
\usepackage{graphicx}
\usepackage{MnSymbol}
\usepackage{tikz}
\usetikzlibrary{positioning}
\usepackage{array}
\usepackage{enumerate}
\usepackage{float}
\usepackage{array}
\usepackage{hyperref}
\newcolumntype{P}[1]{>{\centering\arraybackslash}p{#1}}

\usepackage[ruled,vlined]{algorithm2e}
\usepackage{colortbl}

\usepackage{wrapfig}
\usepackage{booktabs}
\usepackage{subcaption}
\usepackage{lipsum}
\usepackage{adjustbox}
\usepackage{multirow}

\newtheorem{theorem}{Theorem}

\newtheorem{observation}{Observation}

\newtheorem{definition}{Definition}
\newtheorem{example}{Example}

\newtheorem{lemma}{Lemma}

\newtheorem{proposition}{Proposition}

\newtheorem{innercustomgeneric}{\customgenericname}
\providecommand{\customgenericname}{}
\newcommand{\newcustomtheorem}[2]{%
  \newenvironment{#1}[1]
  {%
   \renewcommand\customgenericname{#2}%
   \renewcommand\theinnercustomgeneric{##1}%
   \innercustomgeneric
  }
  {\endinnercustomgeneric}
}

\newcustomtheorem{customthm}{Theorem}

\def\frac#1#2{{\begingroup #1\endgroup\over #2}}
\newcommand\restr[2]{{
  \left.\kern-\nulldelimiterspace 
  #1 
  \right|_{#2} 
  }}

\begin{document}

\title[Krammer Polynomials]{A polynomial invariant for plane curve complements: Krammer polynomials}
\author{MEHMET EMIN AKTA\c{S}}
\address{Department of Mathematics, Florida State University, Tallahassee, Florida 32306}
\email{maktas@math.fsu.edu}
\author{SERDAR CELLAT}
\address{Department of Mathematics, Florida State University, Tallahassee, Florida 32306}
\email{scellat@math.fsu.edu}
\author{HUBEYB GURDOGAN}
\address{Department of Mathematics, Florida State University, Tallahassee, Florida 32306}
\email{gordog@math.fsu.edu}

\subjclass[2010]{Primary 14H30, 20F36; Secondary 14H45}



\keywords{Braid monodromy, $n$-gonal curves, Krammer representation, Krammer polynomial}
\maketitle

\begin{abstract}
We use the Krammer representation of the braid group in Libgober's invariant and construct a new multivariate polynomial invariant for curve complements: Krammer polynomial. We show that the Krammer polynomial of an essential braid is equal to zero. We also compute the Krammer polynomials of some certain $n$-gonal curves. 
\end{abstract}
\section{Introduction}

The study of the topology of algebraic curves has a long history. The main question which has been worked on is ``If $C$ is an algebraic curve in a complex projective plane $P^2$, what is the fundamental group of $P^2-C$?". 

The fundamental group of a complement of a projective plane curve can be studied in terms of a generic projection of the complement to $P^1$ and the braid monodromies around the singular fibers. Zariski-van Kampen theorem gives a way to compute a presentation of the fundamental group of a plane curve from the braid monodromies \cite{van1933fundamental}.


In general, computing the fundamental group of curve complement is not an easy task and it is also hard to distinguish two fundamental groups by their presentations.  In the early 80's, A. Libgober \cite{libgober1982alexander} defined the Alexander polynomial as an invariant of the fundamental group. He, for example, showed that in Zariski's example, the sextic with 6 cusps where the cusps are on a conic has the Alexander polynomial $t^2-t+1$ whereas the other has 1.

In 1989, Libgober defined a more general polynomial invariant based on the braid monodromies using the representations of the braid group \cite{libgober1989invariants}. He also showed that this invariant coincides with the Alexander polynomial when the Burau representation of the braid group is used. He proposed using other representations to get other polynomial invariants, possibly multivariate polynomials \cite{Libgober2005}. 

\subsection{Main Results} In this paper, we use the Krammer representation of the braid group in Libgober's invariant and construct a new multivariate polynomial invariant, \textit{Krammer polynomials}. Our first result is about the Krammer polynomial corresponding to an essential braid:

\begin{theorem}\label{main:thm}\em 
Krammer polynomial of an essential braid $b \in B_n$ is equal to zero.
\end{theorem}

We also study the Krammer polynomials of the $n$-gonal curves. An $n$-gonal curve is an algebraic curve equipped with a pencil of degree $n$ (see Section \ref{ngonal} for more information). Using Theorem \ref{main:thm}, we compute the local Krammer polynomials of $n$-gonal curves around a special type singular fiber:

\begin{theorem}{\em
Let $C$ be a completely reducible $n$-gonal curve and $F$ be its singular fiber where only $m$ components intersect with $m<n$. Then the local Krammer polynomial for the monodromy around $F$ is equal to zero.} 
\end{theorem}

We show that the Krammer polynomial of an $n$-gonal curve is not always zero:

\begin{theorem}\em
The Krammer polynomial $\mathfrak{k}(t,q)$ of a completely reducible $n$-gonal curve $C$ that has one singular fiber only is given by
$$
(t^{2d}q^{6d}-1)^{{n}\choose{2}}
$$
where $d$ is maximum degree of the irreducible components of $C$.

\end{theorem}

\textbf{Organization of the paper.} The paper is structured as follows: In Section \ref{sec:prelim}, we briefly define the $n$-gonal curves and the braid monodromy of the $n$-gonal curves. In Section \ref{sec:kram}, we define the Krammer representation of the braid group and Libgober invariant. We introduce the Krammer polynomials of curve complements in this section. In Section \ref{sec:ess}, we present some results on the Krammer polynomial of essential braids. In \ref{sec:ngon}, we study the Krammer polynomials of $n$-gonal curves and prove our last two main results. We conclude our work in Section \ref{sec:conc}. 

\section{Preliminaries}\label{sec:prelim}

In this chapter, we define the $n$-gonal curves and the braid monodromy of the $n$-gonal curves which is the important tool for computing the invariants of curve complements.
\subsection{The $n$-gonal curves \label{ngonal}}
Let $\Sigma=P^1\times P^1$ and let $p:\Sigma\rightarrow P^1$ be a projection of $\Sigma$ to one of its components. Let $E$ be a section of $p$ and for each $b$ in $P^1$, let $F_b$ be the fiber over $b$.

\begin{definition} \em
An \emph{$n$-gonal curve} is a curve $C\subset \Sigma$ not containing $E$ or a fiber of $\Sigma$ as a component such that the restriction $p: C \rightarrow P^1$ is a map of degree $n$, i.e. each fiber intersects with $C$ in at most $n$ points. In the affine part, $C$ is defined by $F(x,y)=0$ with $F(x,y)\in \mathbb{C}[x,y]$ where deg$_y F=n$. 
\end{definition}

As we understand from previous definition, fibers do not have to intersect with $C\cup E$ at $n+1$ points. A \emph{singular fiber} of a trigonal curve $C\in \Sigma$ is a fiber $F$ of $\Sigma$ intersecting $C\cup E$ geometrically fewer than $n+1$ points. Hence, $F$ is singular either it passes from $C\cap E$, or $C$ is tangent to $F$ or $C$ has a singular point in $F$.  

In this paper, we sometimes narrow our studies for a special subset of $n$-gonal curves. 
\begin{definition}\em{
An $n$-gonal curve $C$ is \textit{completely reducible} if it is defined by $F(x,y)=(y-y_1(x))\cdots(y-y_n(x))=0$ where $y_i \in \mathbb{C}[x]$ for all $i\in \{1,...,n\}$.}
\end{definition}

\subsection{The Braid Monodromy of $n$-gonal Curves}

Let C be an $n$-gonal curve. Let $F_1, F_2, ..., F_r$ be the singular fibers of $C$ and $E$ be the distinguished section. Pick a nonsingular fiber $F$ and let $F^{\sharp} =F\setminus (C\cup E)$. Clearly, $F^{\sharp}$ is equal to $F\setminus E$ with $n$ punctures i.e. it is isomorphic to $n$-punctured complex disk $D_n$. Let $B^{\sharp} = P^1 \setminus \{ p_1, p_2,..., p_r\}$ where $p_i$ is the image under the ruling of the corresponding singular fiber $F_i$. 

We know that $\pi_1(F^{\sharp})= \langle \alpha_1, ..., \alpha_n \rangle$ where $\alpha_i$ is the loop which covers $i$-th intersection of the fiber $F^{\sharp}$ and the $n$-gonal curve $C$ and $\pi_1(B^{\sharp})=\langle \gamma_1,...,\gamma_r \rangle$ where $\gamma_j$ is the loop which covers $p_j$. For each $j = 1,...,r$, dragging the fiber $F$ along $\gamma_j$ and keeping the base point results in a certain automorphism $\mathfrak{m}_j:\pi_1(F^{\sharp})\rightarrow \pi_1(F^{\sharp})$, which is called \textit{the local braid monodromy of $\gamma_j$}. The set of all local braid monodromies $\{\mathfrak{m}_1,...,\mathfrak{m}_r\}$ is called the \textit{global braid monodromy}.

\section{The Krammer Polynomial as a topological invariant}\label{sec:kram}

In this section, we will first present the matrices for the Krammer representation and then construct the Krammer polynomial. (For more information about the representation, see Appendix \ref{App:Krammer})

\subsection{Matrices for the Krammer representation}

The Krammer representation $K(t,q)$ is a representation of the braid group $\emph{B}_n$ in $GL_m (\mathbb{Z}[q^{\pm 1},t^{\pm 1}])=Aut(F_m)$ where $m={{n}\choose{2}}$ and $F_m$ is the free module of rank $m$ over $\mathbb{Z}[q^{\pm 1},t^{\pm 1}]$ \cite{krammer2000braid}. The representation can be formulated as follows:

\begin{center}
$K(\sigma_k)(e_{i,j})=
    \begin{cases}
      tq^2e_{k,k+1} & i=k, j=k+1; \\
      (1-q)e_{i,k} + qe_{i,k+1} & j=k, i<k; \\
      e_{i,k} + tq^{k-i+1}(q-1)e_{k,k+1} & j=k+1, i<k; \\
      tq(q-1)e_{k,k+1} + qe_{k+1,j}, & i=k, k+1<j; \\
      e_{k,j} + (1-q)e_{k+1,j} & i=k+1, k+1<j; \\
      e_{i,j}, & i<j<k \textit{ or } k+1<i<j;\\
      e_{i,j} + tq^{k-i}(q-1)^2e_{k,k+1} & i<k<k+1<j
\end{cases}
$
\end{center}

where $\{e_{i,j}\}_{1 \leq i < j \leq n}$ is the free basis of $F_m$.
\newline

For example, for the braid group $B_3$, here is the matrix representation of its Artin generators:

\begin{center}
$K(\sigma_1)=\left(
\begin{tabular}{c c c}
  $tq^2$ & 0 & 0 \\
  tq(q-1) & 0 & q \\
  0 & 1 & 1-q
 \end{tabular}\right)$
 
 $K(\sigma_2)=\left(
\begin{tabular}{c c c}
  $1-q$ & $q$ & 0 \\
  1 & 0 & $tq^2(q-1)$ \\
  0 & 0 & $tq^2$
 \end{tabular}\right)$
\end{center}

\subsection{The Krammer Polynomial}
In \cite{libgober1989invariants}, Libgober defined a polynomial invariant employing the representations of the braid group as follows: Let $C$ be an algebraic curve and $\{p_1,...,p_r\}$ be the set of its singularities. Let $\rho$ be a $d$ dimensional linear representation of the braid group $\mathbb{B}_n$ over the ring $A$ of Laurent polynomials $\mathbb{Q}[t_1^{\pm 1},...,t_k^{\pm 1}]$ for $k \in \mathbb{N}$	. 
\begin{definition}\em
The \textit{Libgober invariant}, $P(C,\rho)$, is the greatest common divisor of the order $d$ minors in the $N\times d$ matrix of the map $\bigoplus (\rho(\mathfrak{m}(\gamma_i))-$Id$)$, where $N=rd$, $\gamma_i$ is the loop encloses $p_i$ and $\mathfrak{m}(\gamma_i)$ is the braid monodromy of the loop $\gamma_i$. We call this matrix the \emph{Libgober matrix}. It takes $(A^d)^{N}$ to $(A^d)$. 
\end{definition}

Now, we can define the Krammer polynomial.

\begin{definition}\em
If we take $\rho$ as the Krammer representation $K$ of the braid group in Libgober's invariant, then $P(C,K)$ is called the \textbf{Krammer polynomial $\mathfrak{k}(t,q)$.}
\end{definition}

Here is an example for a Krammer polynomial of a trigonal curve.

\begin{example}
Let $C:(y-x^3)(y+x^3)(y-4x)$ be a trigonal curve. It has the singular fiber $x=0$. Using the algorithm in \cite{aktas2017}, we find the corresponding braid monodromy as follows:
$$
\sigma_1\sigma_2\sigma_1(\sigma_2)^4\sigma_1\sigma_2\sigma_1
$$

The corresponding Libgober matrix $L_C$ is 
\begin{center}
$\left(
\begin{tabular}{c c c}
  $t^6q^{14} - 1$ & 0 & 0 \\
  $t^3q^8(q - 1)(t^3q^5 + tq^2 - q + 1)$ & $ t^2q^9 - t^2q^8 + t^2q^7 - 1  $ & $-t^2q^7\frac{q^4-1}{q+1}$ \\
  $t^3q^7(q - 1)(t^2q^4 - tq^3 + tq^2 + q^2 - q + 1) $ & $ t^2q^6(q - 1)(q^2 + 1) $ & $t^2q^{6}\frac{q^5+1}{q+1}-1$
 \end{tabular}\right)$

\end{center}

The Krammer polynomial is given by the greatest common divisor of the order $3$ minors in $L_C$. Since we just compute the local Krammer polynomial around the singular fiber $x=0$, it is equal to the determinant of $L_C$ which gives
$$
\mathfrak{k}_C(t,q)=(t^6q^{14}-1)(t^2q^{10}-1)(t^2q^6-1).
$$

\end{example}

\section{Krammer Polynomial of essential braids}\label{sec:ess}
In this section, we present some results on the Krammer polynomials of essential braids. Our main result here is that the Krammer polynomial of an essential braid is zero.

We first define what an essential braid is.

\begin{definition}\em
A braid element $b$ in the braid Group $B_n$ is an \textit{essential braid} if it does not have at least one of the generators in it. For example, $b=\sigma_1\sigma_2\sigma_4 \in B_5$ is essential since $\sigma_3$ is not in $b$.
\end{definition}

Now, we introduce some useful observations about the Krammer representations of essential braids. Here $\sigma_i^k$ denotes the artin generator $\sigma_i$ of the braid group $\mathbb{B}_k$ for $1\leq i \leq k-1$. 

\begin{observation}\label{obs:1}\em
For $2\leq i \leq n-1$,
$$K(\sigma_i^n) =\left(
\begin{tabular}{c|c}
  $\beta_i^{n-1}$ & $\gamma_i^{n-1}$ \\
  \hline
  $0$ & $\sigma_{i-1}^{n-1}$ \\
\end{tabular}\right)$$
where $\gamma_i^{n-1}$ is an $n-1 \times \frac{(n-1)(n-2)}{2}$ matrix with all zero entries but last $n-i$ entries of the $\hat{k}$th column, with $\hat{k}=1+\sum_{j=3}^i(n-j+1)$, are
$$\left(
\begin{tabular}{c}
  $tq^i(q-1)$\\
  $tq^{i-1}(q-1)^2$ \\
  \vdots \\
  $tq^{i-1}(q-1)^2$ 
 \end{tabular}\right)$$
and
$$\beta_i^{n-1}= \left(
\begin{tabular}{c|c c|c}
  $I_{i-2}$ & 0 & 0 & 0\\
  \hline
  0 & 1-q & q & 0\\
  0 & 1 & 0 & 0\\
  \hline
  0 & 0 & 0 & $I_{n-i-1}$
 \end{tabular}\right).$$
\end{observation}

Now we have a proposition about $\beta_i^{n-1}$ of $\mathbb{B}^n$:

\begin{proposition}\label{prop:1}\em
Let $\beta_i^{n-1}$ of $\mathbb{B}^n$ be defined as in Observation \ref{obs:1} and let $\beta= \prod_{k=1}^r\beta_{i_k}^{n-1}$ with $i_k \in \{2,...,n-1\}$ for all $k$. Then $det(\beta-I_{n-1})=0$.
\end{proposition}
\begin{proof}
In general, the equation $det(\beta - \lambda I)=0$ is equivalent to the statement that $\lambda$ is an eigenvalue of $\beta$. In our case $\lambda=1$. To prove $\lambda=1$ is an eigenvalue, we only need to show that there exists a vector $v$ such that $\beta v=v$. Take $v$ be the all ones vectors, i.e. the vector whose entries are all equal to 1. Then it is clear that $\beta_i v=v$ for all $i$. Hence, $det(\beta - I_{n-1})=0$.
\end{proof}

\begin{observation}\label{obs:6}\em
For $2\leq i \leq n-1$, the non-zero columns of $\gamma_i^{n-k}$ for all $k\in \{1,...,n-1\}$ add up into the $\hat{k}$th column of $K(\sigma_i)$ with $\hat{k}=1+\sum_{j=1}^{k-1}(n-j)$. This column is called the \textit{non-trivial} column of $K(\sigma_i)$.
\end{observation}

\begin{example}
In $\mathbb{B}^6$, the non-trivial columns in $K(\sigma_i)$ for $i\in \{1,2,3,4,5\}$ is given in the Table \ref{table:obs2}.

\begin{center}
\begin{table}
\caption{The non-trivial columns in the Krammer representation of the artin generators of $\mathbb{B}^6$ }
\label{table:obs2}
\begin{tabular}{|c|c|c|c|c|}
\hline 
$\sigma_1$ & $\sigma_2$ & $\sigma_3$ & $\sigma_4$ & $\sigma_5$ \\ \hline\hline
$tq^2$ &  0& 0 &  0&0 \\
$tq(q-1)$ & $tq^2(q-1)$ & 0 & 0 &0\\
$tq(q-1)$ & $tq(q-1)^2$ & $tq^3(q-1)$ &0 & 0\\
$tq(q-1)$ & $tq(q-1)^2$ & $tq^2(q-1)^2$ & $tq^4(q-1)$ & 0\\
$tq(q-1)$ & $tq(q-1)^2$ & $tq^2(q-1)^2$ & $tq^3(q-1)$ & $tq^5(q-1)$ \\
 0& $tq^2$ & 0 & 0 &0 \\
0 & $tq(q-1)$ & $tq^2(q-1)$ &0 &0\\
0& $tq(q-1)$ & $tq(q-1)^2$ & $tq^3(q-1)$ &0 \\
0& $tq(q-1)$ & $tq(q-1)^2$ & $tq^2(q-1)^2$ & $tq^4(q-1)$ \\ 
0&0 & $tq^2$ & 0 &0\\
0&0 & $tq(q-1)$ & $tq^2(q-1)$ & 0 \\
0& 0&  $tq(q-1)$ & $tq(q-1)^2$ & $tq^3(q-1)$ \\ 
0&0 & 0&  $tq^2$ &0\\
0&0 & 0&  $tq(q-1)$ & $tq^2(q-1)$\\ 
0& 0&0 &0 & $tq^2$ \\ \hline
 \end{tabular}
\end{table}
\end{center}

\end{example}
\begin{observation}\label{obs:2}\em
For $1\leq i \leq n-2$, 
$$K(\sigma_i^n) =E\left(
\begin{tabular}{c|c}
  $\alpha_i^{n-1}$ & $\eta_i^{n-1}$ \\
  \hline
  $0$ & $\sigma_{i}^{n-1}$ \\
\end{tabular}\right)E$$
where $\eta_i^{n-1}$ is an $n-1 \times \frac{(n-1)(n-2)}{2}$ matrix, $E$ is a fixed elementary unitary matrix that switches $[\sum_{i=1}^k(n-i)]'$th row and column with $[k+1]'$st row and column respectively for all $k\in \{0,...,n-1\}$ and
$$\alpha_i^{n-1}= \left(
\begin{tabular}{c|c c|c}
  $I_{i-1}$ & 0 & 0 & 0\\
  \hline
  0 & 0 & q & 0\\
  0 & 1 & 1-q & 0\\
  \hline
  0 & 0 & 0 & $I_{n-i-2}$
 \end{tabular}\right).$$
\end{observation}

\begin{proposition}\label{prop:2}\em
Let $\alpha_i^{n-1}$ of $\mathbb{B}^n$ is defined as in Observation \ref{obs:2} and let $\alpha= \prod_{k=1}^r\alpha_{i_k}^{n-1}$ with $i_k \in \{2,...,n-1\}$ for all $k$. Then $det(\alpha-I_{n-1})=0$.
\end{proposition}
\begin{proof}
Similar to the Proposition \ref{prop:1}, to prove that $det(\alpha - I_{n-1})=0$, it suffices to show that there exists a vector $v$ such that $\alpha v=v$ or $\alpha^t v= v$. Let $v$ be the all ones vectors. Then it is clear that $\alpha_i^t v=v$ for all $i$. Hence, $det(\alpha^t - I_{n-1})=det(\alpha - I_{n-1})=0$.
\end{proof}







  


\begin{observation}\label{obs:5}\em
For $1\leq i \leq n-2$,
$$K(\sigma_i^n) =\left(
\begin{tabular}{c|c|c|c}
  $\diamond$ & 0 & 0 & 0 \\
  \hline
  * & 0 & $qI_{n-i-1}$ & 0 \\
  \hline
  0 & $I_{n-i-1}$ & $(1-q)I_{n-i-1}$ & 0 \\
  \hline
  0 & 0 & 0 & $I_{n'}$ \\
  
\end{tabular}\right)$$
where $\diamond$ is an $n_1 \times n_1$ matrix with $n_1=1+\sum_{j=1}^{i-1}(n-j)$, $*$ is an $(n-i-1) \times n_1$ matrix and $n'=n-2i+1$.
\end{observation}

Now, we prove our main result in this section.



\begin{customthm}{1}\label{main:thm}\em 
Krammer polynomial of an essential braid $b \in B_n$ is equal to zero.
\end{customthm}

\begin{proof} We will divide the proof into three cases.
\begin{enumerate}
    \item $b$ does not have the generator $\sigma_1$.
    \item $b$ does not have the generator $\sigma_{n-1}$.
    \item $b$ does not have the generator $\sigma_i$ for $1<i<n-1$.
\end{enumerate}

We prove each case separately.

\textbf{Case (1):} Assume that $b$ does not have the generator $\sigma_1$, then $ b \in <\sigma_2^n,...,\sigma_{n-1}^n> \subset \mathbb{B}_n$. Let $b=\sigma_{i_1}^{n}...\sigma_{i_r}^{n}$ be an essential braid in $\mathbb{B}_n$ where $2\leq i_j \leq n-1$ for all $j \in \{1,...,r\}$. From Observation \ref{obs:1}, \begin{center}

$b =$ $\left(
\begin{tabular}{c|c}
  $\beta_{i_1}^{n-1}$ & $\gamma_{i_1}^{n-1}$ \\
  \hline
  $0$ & $\sigma_{i_1-1}^{n-1}$ \\
\end{tabular}\right) \cdots \left(
\begin{tabular}{c|c}
  $\beta_{i_r}^{n-1}$ & $\gamma_{i_r}^{n-1}$ \\
  \hline
  $0$ & $\sigma_{i_r-1}^{n-1}$ \\
\end{tabular}\right) = \left(
\begin{tabular}{c|c}
  $\prod_{k=1}^r\beta_{i_k}^{n-1}$ & $\gamma$ \\
  \hline
  $0$ & $\prod_{k=1}^r\sigma_{i_k-1}^{n-1}$ \\
\end{tabular}\right)$
\end{center}
where $\gamma$ is an $n-1 \times \frac{(n-1)(n-2)}{2}$ matrix. Hence, the Krammer polynomial of $b$ is
   
\begin{align*}
\mathfrak{k}(b)= det(K(b)-I) &= det\left(
\begin{tabular}{c|c}
  $\prod_{k=1}^r\beta_{i_k}^{n-1} -I$ & $\gamma$ \\
  \hline
  0 & $\prod_{k=1}^r\sigma_{i_k-1}^{n-1}-I$ \\
\end{tabular}\right)\\ 
&= det(\prod_{k=1}^r\beta_{i_k}^{n-1}-I)det(\prod_{k=1}^r\sigma_{i_k-1}^{n-1}-I)\\ 
&=0 \\
\end{align*}

since $det(\prod_{k=1}^r\beta_{i_k}^{n-1}-I)=0$ from Proposition \ref{prop:1}.

\textbf{Case (2):} Assume that $b$ does not have the generator $\sigma_{n-1}$, then $ b \in <\sigma_1^n,...,\sigma_{n-2}^n> \subset \mathbb{B}_n$. Let $b=\sigma_{i_1}^{n}...\sigma_{i_r}^{n}$ be an essential braid in $B_{n}$ where $1\leq i_j \leq n-2$ for all $j \in \{1,...,r\}$. In Observation \ref{obs:2}, since $E$ is an elementary unitary matrix, we have
    
    $K(b) =$ $E\left(
\begin{tabular}{c|c}
  $\alpha_{i_1}^{n-1}$ & $\eta_{i_1}^{n-1}$ \\
  \hline
  $0$ & $\sigma_{i_1}^{n-1}$ \\
\end{tabular}\right)E \cdots E\left(
\begin{tabular}{c|c}
  $\alpha_{i_r}^{n-1}$ & $\eta_{i_r}^{n-1}$ \\
  \hline
  $0$ & $\sigma_{i_r}^{n-1}$ \\
\end{tabular}\right)E = E\left(
\begin{tabular}{c|c}
  $\prod_{k=1}^r\alpha_{i_k}^{n-1}$ & $\eta$ \\
  \hline
  $0$ & $\prod_{k=1}^r\sigma_{i_1}^{n-1}$ \\
\end{tabular}\right)E$
where $\eta$ is an $n-1 \times \frac{(n-1)(n-2)}{2}$ matrix. Hence, the Krammer polynomial of $b$ is
   
   \begin{align*}
\mathfrak{k}(b)= det(K(b)-I) &= det(E\left(
\begin{tabular}{c|c}
  $\prod_{k=1}^r\alpha_{i_k}^{n-1}$ & $\eta$ \\  \hline
  0 & $\prod_{k=1}^r\sigma_{i_1}^{n-1}$ \\
\end{tabular}\right)E - I)\\ &= det(E)det(\left(
\begin{tabular}{c|c}
  $\prod_{k=1}^r\alpha_{i_k}^{n-1}$ & $\eta$ \\ \hline
  0 & $\prod_{k=1}^r\sigma_{i_1}^{n-1}$ \\
\end{tabular}\right)-I)det(E)\\ &=det(\left(
\begin{tabular}{c|c}
  $\prod_{k=1}^r\alpha_{i_k}^{n-1}$ & $\eta$ \\ \hline
  0 & $\prod_{k=1}^r\sigma_{i_1}^{n-1}$ \\
\end{tabular}\right)-I)\\ &=det(\prod_{k=1}^r\alpha_{i_k}^{n-1}-I)det(\prod_{k=1}^r\sigma_{i_1}^{n-1}-I)\\ &=0
\end{align*}
since $det(\prod_{k=1}^r\alpha_{i_k}^{n-1}-I)=0$ from Proposition \ref{prop:2}.

\textbf{Case (3):} Assume that $b$ does not have the generator $\sigma_i$ for $1<i<n-1$ i.e. $b \in <\sigma_1,...,\sigma_{i-1},\sigma_{i+1},...,\sigma_{n-1}> \subset \mathbb{B}_n$. 

Again, similar to the proof of Proposition \ref{prop:1} and \ref{prop:2}, we only need to show that there exists a vector $v$ such that $v K(b)=v$ (i.e. $v$ is a left-eigenvector with size $m={{n}\choose{2}})$. Hence, it is enough to show that there exist a vector $v$ such that $v K(\sigma_k) = v$ for all $k \in \{1,...,i-1,i+1,...,n-1\}$. Let $v=[v_1 \cdots v_m]$, where $v_j \in  \mathbb{Z}[q^{\pm 1},t^{\pm 1}]$ and $m=\frac{n(n-1)}{2}$, be a vector. Let us further partition $v$ as 
$$
v = [ v_1,...,v_m] =[[\delta_1],...,[\delta_{n-1}]]
$$ 
where $\delta_k$ is a sub matrix of $v$ with $n-k$ elements for $1\leq k \leq n-1$. This partitioning reveals the relationship among the elements of $v$ and makes its structure more clear. Moreover, we use the notation $\delta_k^j$ for the $j$th element in $\delta_k$. 

From the identity block matrix of $K(\sigma_k)$ in Observation \ref{obs:5}, we notice that some elements of $v$ are equal to each other. That is, 
\begin{equation}\label{rel1}
\delta_1^j = \delta_k^{j-k+1}
\end{equation}for $k\leq i$ and $j \in \{2,...,n-1\}$. For $1<k<n-i-1$, we observe that 
\begin{equation}\label{rel2}
\delta_{i+1}^j =\delta_{i+k}^{j-k+1}
\end{equation}
where $j \in \{2,...,n-i-1\}$. Moreover, from the block matrix $\left(\begin{tabular}{c c}
  1-q & q \\
  1 & 0
\end{tabular} \right)$ in Observation \ref{obs:1}, we get 
\begin{equation}\label{rel3}
\delta_k^{j+1} = q\delta_k^j 
\end{equation} 
for all $k\in \{1,...,n-1\}$ and $j\leq n-k$.

Combining the relations (\ref{rel1}), (\ref{rel2}) and (\ref{rel3}), we obtain 3 different cases for $\delta_k$ matrices, 

\begin{equation}\label{delta_k}
\delta_k =
    \begin{cases}
      [xq^{k-1},...,xq^{i-2},1,q,\cdots,q^{n-i-1}]  & k<i; \\
      [1,q,q^2,q^3,...,q^{n-i-1}] & k=i; \\
      [yq^{k-i-1},...,yq^{n-i-2}] & k>i
      \end{cases}
\end{equation}
where $x$ and $y$ are two indeterminants. Although $\delta_k^{i-k+1}=\delta_i^1$ for $k<i$ can be any real number, 1 is chosen for simplicity.
\begin{example}
For $n=6$ and $i=3$ (i.e. the braid is in $\mathbb{B}^6$ and does not have the generator $\sigma_3$), the vector $v$ is given by
$$v=[x,xq,xq^2,1,q,xq,xq^2,1,q,xq^2,1,q,q^2,y,yq,yq].$$
\end{example}
Furthermore, the $\hat{k}$th column of $K(\sigma_k)$ with $\hat{k}=1+\sum_{j=1}^{k-1}(n-j)$, which is already mentioned in Observation \ref{obs:6}, provides extra relation $f_k$ that can be employed to find $x$ and $y$. Indeed, there exist $i-1$ relations with entries in $[[\delta_1],...,[\delta_{i-1}]]$ that include $x$ and $n-i-1$ linear relations with entries in $[[\delta_{i+1}],...,[\delta_{n-1}]]$ that include $y$. More specifically, $f_k$ has entries from the last $n-k$ elements of $\delta_k$ for all $k\in \{1,...,i-1,i+1,...,n-1\}$. For example, in $B_6$, $f_2$ has entries $\{\delta_1^2,...,\delta_1^5, \delta_2^1,...,\delta_2^4\}$ where $\delta_k^j$ is the $j$-th element in $\delta_k$.

Note that all other rows in $K(\sigma_k)$ give redundant relations i.e. the relations (\ref{rel1}), (\ref{rel2}), (\ref{rel3}) and $f_k$ for $k\in \{1,...,i-1,i+1,...,n-1\}$ are all relations that the elements of the eigenvector $v$ need to satisfy. 


\begin{lemma}\em\label{lemma1}
There exist unique solutions for $x$ and $y$ in (\ref{delta_k}) that the relations $f_k$ for $k\in \{1,...,i-1,i+1,...,n-1\}$ are satisfied.
\end{lemma}
\begin{proof}
First, it is easy to check that $f_k$ with $k<i$ gives the relations for $x$ and with $k>i$  for $y$. For $k<i$, $f_k$ gives 
$$
xtq^{i+k-1}-tq^{k-1}(1-q^{n-i})=xq^{k-1}.
$$
This yields to the following solution:
$$
x=\frac{tq^{k-1}(1-q^{n-i})}{q^{k-1}(tq^i-1)}=\frac{t(1-q^{n-i})}{tq^i-1}
$$
which is independent from $k$, i.e. there is a unique solution for $x$ for all $f_k$ with $k<i$.

Similarly, for $k>i$, $f_k$ gives 
$$
ytq^{k-1}-tq^{k-1}(1-q^{n-i})=yq^{k-i-1}.
$$
Hence, we have the following solution for $y$:
$$
y=\frac{tq^{k-1}(1-q^{n-i})}{q^{k-i-1}(tq^i-1)}=\frac{tq^{i+1}(1-q^{n-i})}{tq^i-1}.
$$
which is again independent from $k$, i.e. there is a unique solution of $y$ for all $f_k$ with $k>i$. 
\end{proof}
Now, we give an example for the construction in the lemma above.
\begin{example}
Let $n=6$ and $i=3$ again. The vector $v$ and the non-trivial column for each generator is given by Table \ref{table:2}.

From the non-trivial column of $K(\sigma_1)$, we get the relation $f_1$ as
$$
xtq^2+xtq^2(q-1)+(1+q+q^2)tq(q-1) =x \Rightarrow xtq^3-tq(1-q^3)=x \Rightarrow \\
x(tq^3-1) = tq(1-q^3),
$$
which implies
$$
x = \frac{tq(1-q^3)}{tq^3-1}.
$$
Similarly, the non-trivial column of $K(\sigma_2)$ gives the relation $f_2$ as
$$
xtq^3(q-1)+(1+q+q^2)tq(q-1)^2+xtq^3 + (1+q+q^2)tq(q-1) = xq,
$$
which gives 
$$
xtq^4-tq^2(1-q^3) = xq.
$$
We get the same relation when we cancel the $q$'s on both sides of the equation, i.e. $f_2$ gives the same solution for $x$.

We obtain the unique solution for $y$ using the non-trivial columns of $K(\sigma_4)$ and $K(\sigma_5)$ similarly.
\begin{center}
\begin{table}
\caption{The non-trivial columns in the Krammer representation of the artin generators for $\mathbb{B}^6$ }
\label{table:2}
\begin{tabular}{|c|c|c|c|c|c|c|}
\hline 
 \multicolumn{2}{|c|}{$v$} & $\sigma_1$ & $\sigma_2$ & $\sigma_3$ & $\sigma_4$ & $\sigma_5$ \\ 
\hline \hline
\multirow{5}{*}{$\delta_1$} & $x$ & $tq^2$ &  0& 0 &  0&0 \\
 & $xq$ & $tq(q-1)$ & $tq^2(q-1)$ & 0 & 0 &0\\
 & 1 & $tq(q-1)$ & $tq(q-1)^2$ & $tq^3(q-1)$ &0 & 0\\
 & $q$ & $tq(q-1)$ & $tq(q-1)^2$ & $tq^2(q-1)^2$ & $tq^4(q-1)$ & 0\\
 & $q^2$ & $tq(q-1)$ & $tq(q-1)^2$ & $tq^2(q-1)^2$ & $tq^3(q-1)$ & $tq^5(q-1)$ \\ \hline
\multirow{4}{*}{$\delta_2$} & $xq$ &0& $tq^2$ & 0 & 0 &0 \\
 & $1$ & 0 & $tq(q-1)$ & $tq^2(q-1)$ &0 &0\\
 & $q$ & 0& $tq(q-1)$ & $tq(q-1)^2$ & $tq^3(q-1)$ &0 \\
 & $q^2$ & 0& $tq(q-1)$ & $tq(q-1)^2$ & $tq^2(q-1)^2$ & $tq^4(q-1)$ \\ \hline
\multirow{3}{*}{$\delta_3$} & 1 & 0&0 & $tq^2$ & 0 &0\\
& $q$ & 0&0 & $tq(q-1)$ & $tq^2(q-1)$ & 0 \\
& $q^2$ & 0& 0&  $tq(q-1)$ & $tq(q-1)^2$ & $tq^3(q-1)$ \\ \hline
\multirow{2}{*}{$\delta_4$} & y & 0&0 & 0&  $tq^2$ &0\\
& $yq$ & 0&0 & 0&  $tq(q-1)$ & $tq^2(q-1)$\\ \hline
$\delta_5$& $yq$ & 0& 0&0 &0 & $tq^2$ \\ \hline
 \end{tabular}
\end{table}
\end{center}

\end{example}

From Lemma \ref{lemma1}, we finished the construction of the vector $v$. To sum up, $$v=[[\delta_1],...,[\delta_{n-1}]]$$ where $$\delta_k =
    \begin{cases}
      [xq^{k-1},...,xq^{i-2},1,q,\cdots,q^{n-i-1}]  & k<i; \\
      [1,q,q^2,q^3,...,q^{n-i-1}] & k=i; \\
      [yq^{k-i-1},...,yq^{n-i-2}] & k>i
      \end{cases}$$
      with $$x=\frac{t(1-q^{n-i})}{tq^i-1} \text{ and } y=\frac{tq^{i+1}(1-q^{n-i})}{tq^i-1}.$$ This completes the proof. 
\end{proof}

\section{Krammer polynomial of the $n$-gonal curves}\label{sec:ngon}

In this section, we compute the Krammer polynomials of two sets of $n$-gonal curves. We already mentioned these results in introduction. Here, we will prove these computations.

First, we compute the local Krammer polynomial around a special singular fiber.

\begin{customthm}{2}\em
The local Krammer polynomial for the monodromy around a singular fiber $F$ of an $n$-gonal curve where $m<n$ components intersect is equal to zero.
\end{customthm}
\begin{proof}
Since $m<n$ components intersect over $F$, $n-m$ strands of the corresponding braid monodromy is fixed. In other words, the braid monodromy gives an essential braid. Hence, from Theorem \ref{main:thm}, we can deduce that the Krammer polynomial around the fiber $F$ is equal to zero  
\end{proof}



In the second result, we compute the global Krammer polynomials of the a completely reducible $n$-gonal curves that have one singular fiber only. 

\begin{customthm}{3}{\em
The global Krammer polynomial of a completely reducible $n$-gonal curve $C$ that has one singular fiber only is
$$
(t^{2d}q^{6d}-1)^{{n}\choose{2}}
$$
where $d$ is maximum degree of the irreducible components of $C$. }
\end{customthm}

\begin{proof}
Since $C$ has only one singular fiber, it is in the form $(y+a_1(x-p_1)^d+p_2)...(y+a_n(x-p_1)^d+p_2)$ where $a_i\neq a_j$ $\forall i\neq j$, $p_1,p_2 \in \mathbb{C}$ and $d\in \mathbb{Z}$. In other words, all the irreducible components of $C$ intersects at $x=p_1$. In this case, there are $d$ full twists around the singular fiber $F_{p_1}$. Hence, the local monodromy around this fiber is ${(\sigma_1\sigma_2)}^{d}$. Then, we have the following $m \times m$ local Libgober matrix where $m={{n}\choose{2}}$:

\begin{center}
$L_C=\left(
\begin{tabular}{c c c}
  $t^{2d}q^{6d}-1$ & $\cdots$ & 0 \\
  0 & $\ddots$ & 0 \\
  0 & $\cdots$ & $t^{2d}q^{6d}-1$
 \end{tabular}\right)$

\end{center}

Furthermore, this matrix is the global Libgober matrix since there is only one singular fiber. Thus, the Krammer polynomial $\mathfrak{k}_C(t,q)$ is equal to the greatest common divisor of the of the order $m$ minors in $L_C$ which is

$$
\mathfrak{k}_C(t,q)=(t^{2d}q^{6d}-1)^m.
$$

\end{proof}



 

\section{Conclusion}\label{sec:conc}
In this paper, we introduce a new polynomial invariant using Krammer representation of braid group. We compute the Krammer polynomials of essential braids and some certain $n$-gonal curves. As a future task, it would be very interesting to find out other curves which have non-trivial Krammer polynomials. Moreover, following the same idea, we plan to use different representations of the braid group to generate different polynomial invariants. One can use, for example, the Gassner Representation of the pure braid group since the braids in our case are also pure braids.   

\bibliographystyle{plain}
\bibliography{myrefs}

\appendix
\section{The Braid Group $B_n$}

Let $D$ be an oriented disk in the complex plane. Let 
$$
G=\{(z_1,...,z_n) | z_i\in D, z_i\neq z_j \text{ if } i\neq j\}.
$$
The symmetric group $S_n$ acts on the entries of $G$ hence we can define $\hat{G}=G/S_n$ i.e. $\hat{G}$ is the set of all unordered $n$-tuples of $D$ where $z_i\neq z_j$ if $i\neq j$.

\begin{definition}
The fundamental group $\pi_1(\hat{G})$ is called the braid group $B_n$ on $n$ strands. 
\end{definition}
We can define the braid group more geometrically as follows: A \textit{geometric braid} on n strands, where $n\in \mathbb{N}$, is an injective map $\beta: I \times \{1,...,n\} \rightarrow D \times I$ with the following properties:
\begin{itemize}
\item the $x$-coordinate of $\beta(x, k)$ equals $x$ for all $x\in I$ and $k\in \{1, ...,n\}$
\item $\beta(0,k)=(0,k)$ and for each $k\in \{1, ...,n\}$, there is a $k'\in \{1, ...,n\}$ such that $\beta(1,k)=\beta(1,k')$
\end{itemize}
\begin{figure}[ht!]
\centering
\includegraphics[width=70mm]{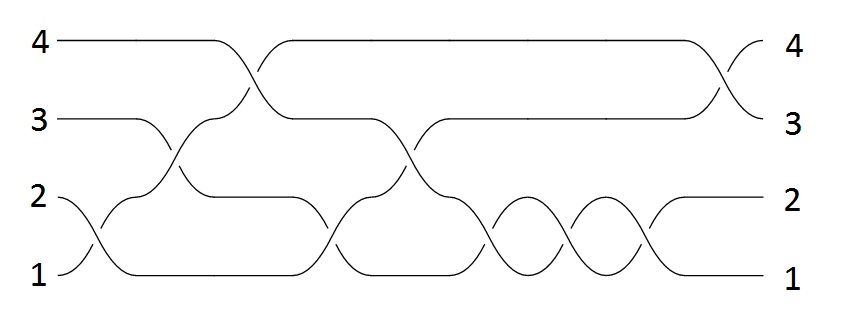}
\caption{A geometric braid}
\label{overflow}
\end{figure}
Two geometric braids are said to be \textit{equivalent} if they are isotopic in the class
of geometric braids. The product $\beta_1\beta_2$ of two geometric braids $\beta_1, \beta_2$ is defined as:
\begin{center}
$\beta_1\beta_2=
    \begin{cases}
      \beta_1(2x,k), & x \leq 1/2	 \\
      \beta_1(2x-1,k), & x \geq 1/2       
\end{cases}
$
\end{center}
and the inverse $\beta^{-1}$ of a geometric braid is defined as:
$$
\beta^{-1}(x,k)=\beta(1-x, k).
$$
Any strand in a geometric braid $\beta$ gives an element $\hat{\alpha}$ of $\pi_1(C)$ and an equivalence between two geometric braids $\beta, \beta'$ gives a homotopy between $\alpha$ and $\alpha'$. Also product of two braids corresponds to the composition of the corresponding loops. Therefore, alternatively, we can define the braid group $B_n$ as the collection of geometric braids modulo equivalence. 

\textbf{Generators and Relations of $B_n$}

Let $\sigma_i \in B_n$ where $\sigma_i$ twists the $i$-th and $i+1$-st strands through an angle of $\pi$ in the counterclockwise direction while leaving the other strands fixed, see Figure \ref{generbraid}. It is clear that $\sigma_1$'s satisfy the following relations:
\begin{equation}
\sigma_i\sigma_{j}=\sigma_{j}\sigma_i 	\text{ if } |i-j|> 1, \hspace{5mm}\sigma_i\sigma_{i+1}\sigma_i=\sigma_{i+1}\sigma_i\sigma_{i+1} 
\label{relations}
\end{equation}

\begin{figure}[ht!]
\centering
\includegraphics[width=100mm]{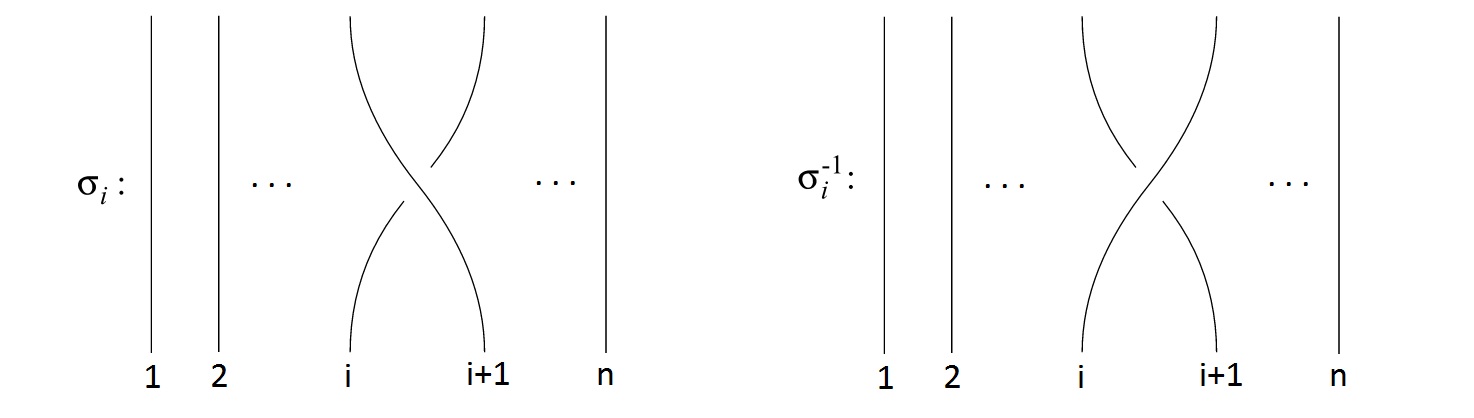}
\caption{Generators of $B_n$}
\label{generbraid}
\end{figure}

In the following theorem, Artin showed these braids and relations are sufficient to define $B_n$.
\begin{theorem}[Artin \cite{artin1947theory}]
The presentation of the braid group $B_n$ has generators $\sigma_1,...,\sigma_{n-1}$ and relations given by \ref{relations}.
\end{theorem}

\textbf{$B_n$ as the group of automorphism}

Let $P=\{p_1,...,p_n\}$ be a set of $n$ points in the interior of $D$ and $D_n=D \setminus P$ i.e. $D_n$ is a punctured disk. The fundamental group of $D_n$ on a base point on the boundary of $D$ is isomorphic to the free group with $n$ generators. Let ${F}_n=\langle \alpha_1, ... , \alpha_n \rangle$ be the free group on $n$ generators. The \emph{braid group} ${B}_n$ can be defined as the group of automorphisms $\beta: {F}_n\rightarrow {F}_n$ with the following properties:
\begin{itemize}
 \item each generator $\alpha_i$ is taken to a conjugate of a generator;
 \item the element $\rho:= \alpha_1...\alpha_n$ remains fixed.
\end{itemize}
and the action of ${B}_n$ on ${F}_n$ as follows:
\begin{center}
$\sigma_i(\alpha_j)=
   \begin{cases}
     \alpha_{j} & j \neq i,i+1 \\
     \alpha_i\alpha_{i+1}{\alpha_i}^{-1} & j=i \\
     \alpha_i & j=i+1
\end{cases}
$
\end{center}

\section{Krammer Representation of the Braid Group\label{App:Krammer}}
One of the popular question in late 20th century was whether the Braid group is linear, i.e. it is isomorphic to a subgroup of $GL(n,K)$ for $n \in \mathbb{N}$ and some field $K$.  It is easy to show that the Burau representation is faithful for $n\leq 3$. Bigelow showed that it is unfaithful for $n\geq 5$ \cite{bigelow1999burau} and it is still unknown that the Burau representation is faithful or not for $n=4$.

Another representation, introduced by Lawrence \cite{lawrence1990homological}, is studied to prove that the Braid group is linear. Krammer showed that this representation is faithful for the Braid group $B_4$ \cite{krammer2000braid}. Then finally Bigelow showed that it is faithful for all $n$ \cite{bigelow2001braid}.

The representation that is mentioned above is called \textit{Krammer representation}. It is in $GL_m\mathbb{Z}[q^{\pm 1},t^{\pm 1}]$ where $m={\frac{n(n-1)}{2}}$ \cite{krammer2000braid}. Here we will present the definition of the representation. 

Let $D_n=D\setminus P$ where $D$ is the complex disk and $P=\{p_1,...,p_n\}$ with $p_i \in D$. Let $L=D_n\times D_n \setminus \Delta$ where $\Delta$ denotes the diagonal of $D_n\times D_n$. Let $K$ be the set of all unoredered pairs of distinct points in $D_n$ i.e. $K=L/S_2$. Let $\alpha:I\rightarrow K$ be a path in $K$ based at $k_0=\{p_0,q_0\}$ where $p_0$ and $q_0$ are points on the boundary of $D_n$. Since the projection $L \rightarrow K$ is a $2$-fold covering, we can lift $\alpha$ to $L$, say $\alpha'$, and it is in the form $\alpha'=(\alpha_1',\alpha_2')$ where $\alpha_i': I \rightarrow D_n$ for $i\in \{1,2\}$ and $\alpha_1'(s)\neq \alpha_2'(s)$ for any $s\in I$. If $\alpha$ is a loop, $\alpha_1'$ and $\alpha_2'$ are either both loops or composition of them results a loop. 

Let $\alpha\in \pi_1(K)$ and define maps $a$ and $b$ from $\pi_1(K)$ to $\mathbb{Z}$ as follows: 
\begin{center}
$a(\alpha)=
    \begin{cases}
      w(\alpha_1')+w(\alpha_2'), & \alpha_1' \text{ and }  \alpha_2' \text{ are both closed loops}	 \\
      w(\alpha_1'\alpha_2'), & \text{otherwise}       
\end{cases}
$
\end{center}
where $w$ denotes the winding number around the puncture points $p_1,...,p_n$ and 
$$b(\alpha)=\beta_2\beta_1$$ 
where $\beta_1:I \rightarrow S^1$ with
\begin{center}
$\beta_1(s)=\frac{\alpha_1'(s)-\alpha_2'(s)} {|\alpha_1'(s)-\alpha_2'(s))|}$
\end{center}
and $\beta_2:S^1 \rightarrow H_1(\mathbb{R}P^1)\cong \mathbb{Z}$ is the induced map of the projection $S^1 \rightarrow \mathbb{R}P^1$. In other words, $a$ counts how many times $\alpha_1'$ and $\alpha_2'$ winds around the puncture points and $b$ counts how many times they wind around each other. Now, define a map $\phi: \pi_1(C)\rightarrow \langle q,t \rangle$ by
 $$\phi(\alpha)=q^{a(\alpha)}t^{-b(\alpha)}.$$
  
Let $\tilde{K}$ be the regular covering of $K$ corresponding to the $\ker \phi$. Let $h$ be a homemorphism of $D_n$ to itself. Clearly, this induces a homemorphism $h_K$ of $K$ to itself by $h_K(\{d_1,d_2\})=\{h(d_1),h(d_2)\}$ and this can be lifted to a homeomorphism $\tilde{h}_K: \tilde{K} \rightarrow \tilde{K}$. Since the group $\langle q,t \rangle$ acts on $\tilde{K}$ as a group of covering transformations, the homology $H_2(\tilde{K})$ is a $\mathbb{Z}[q^{\pm 1},t^{\pm 1}]$-module. Hence, $\tilde{h}$ commutes with the covering transformations $q$ and $t$, and $\tilde{h}$ induces a $\mathbb{Z}[q^{\pm 1},t^{\pm 1}]$-module isomorphism $\tilde{h}^*: H_2(\tilde{K}) \rightarrow H_2(\tilde{K})$. The representation $\kappa: B_n \rightarrow Aut(H_2(\tilde{K}))$ where 
$$\kappa(h)=\tilde{h}^*$$
is called the \textit{Krammer representation} of $B_n$.

\end{document}